\documentclass{article}
\usepackage[T1]{fontenc}
\usepackage[utf8]{inputenc}
\usepackage[nobysame]{amsrefs}
\usepackage{xcolor}
\usepackage{hyperref}
\hypersetup{
  colorlinks = true,
  linkcolor  = red!50!black,
  citecolor  = green!50!black,
  urlcolor   = blue!50!black,
  pdftitle   = {On equivariant asymptotic dimension, Damian Sawicki}
}
\usepackage{enumitem}
\usepackage{upref} 
\usepackage{mathtools} 
\usepackage{amsthm,amsfonts}

\newtheorem{thm}{Theorem}[section]
\newtheorem*{claim}{Claim}
\newtheorem{lem}[thm]{Lemma}
\newtheorem{rem}[thm]{Remark}
\newtheorem{prop}[thm]{Proposition}
\newtheorem{cor}[thm]{Corollary}

\theoremstyle{definition}
\newtheorem{defi}[thm]{Definition}
\newtheorem{longRem}[thm]{Remark}
\newtheorem{longEx}[thm]{Example}
\newtheorem*{thmA}{Theorem A}
\newtheorem*{thmB}{Theorem B}

\DeclareMathOperator{\Cay}{Cay}
\DeclareMathOperator{\id}{id}
\DeclareMathOperator{\diam}{diam}
\DeclareMathOperator{\locdim}{loc\ \!dim}
\DeclareMathOperator{\asdim}{asdim}
\DeclareMathOperator{\eqasdim}{eq-asdim}
\DeclareMathOperator{\heqasdim}{\!-eq-asdim}
\def\feqasdim{\cF\!\heqasdim}
\def\U{{\mathcal U}}
\def\V{{\mathcal V}}
\def\W{{\mathcal W}}
\def\N{{\mathbb N}}
\def\Z{{\mathbb Z}}
\def\F{{\mathbb F}}
\def\cF{{\mathcal F}}
\def\eps{\varepsilon}
\def\vcyc{{\mathcal{VC}\text{yc}}}
\def\dX{{\partial X}}
\def\bX{{\myol{X}}}
\def\oB{{\myol{B}}}
\def\st{\ | \ }
\newcommand{\myol}[2][3]{{}\mkern#1mu\overline{\mkern-#1mu#2}} 
\usepackage{dsfont} 
\newcommand{\one}{{\mathds{1}}}
\newcommand*{\defeq}{\mathrel{\vcenter{\baselineskip0.5ex \lineskiplimit0pt
                     \hbox{\scriptsize.}\hbox{\scriptsize.}}}%
                     =}
\usepackage[titletoc,title]{appendix}

\usepackage{authblk}       
\author{Damian Sawicki\thanks{The author was partially supported by the Foundation for Polish Science with the grant HOMING PLUS Bis/2011-4/6.}
}
\affil{Institute of Mathematics, Polish Academy of Sciences\\ \'Sniadeckich 8, 00-656 Warsaw, Poland\\
www.impan.pl/\~{}dsawicki}
\date{}
\title{On equivariant asymptotic dimension}

\begin{document}

\maketitle

\paragraph{Abstract.}
The work discusses equivariant asymptotic dimension (also known as ``wide equivariant covers'', ``$N$-$\cF$-amenability'' or ``amenability dimension'', and ``$d$-BLR condition'') and its generalisation, transfer reducibility, which are versions of asymptotic dimension invented for the proofs of the Farrell--Jones and Borel conjectures.

We prove that groups of null equivariant asymptotic dimension are exactly virtually cyclic groups. We show that a covering of the boundary always extends to a covering of the whole compactification. We provide a number of characterisations of equivariant asymptotic dimension in the general setting of homotopy actions, including equivariant counterparts of classical characterisations of asymptotic dimension. Finally, we strengthen the result of Mole and R\"uping about equivariant refinements from finite groups to infinite groups.

\paragraph{Mathematics Subject Classification (2010).}\hspace{-0.35cm} 20F65 (18F25, 20F67, 57M07).

\paragraph{Keywords.} Equivariant cover, asymptotic dimension, homotopy action, transfer reducible group, Farrell--Jones conjecture.

%
%

\section*{Introduction}\label{S:intro}

The concept of equivariant asymptotic dimension was introduced by Bartels, L\"uck, and Reich in \cite{coversForHyperbolic}. Proving finiteness of equivariant asymptotic dimension was a major technical step in the proof of the Farrell--Jones conjecture for hyperbolic groups \cite{FarrellJonesForHyperbolic}.
A generalisation of this property -- transfer reducibility (see Section \ref{sec:characterisations}) -- was used to prove the Farrell--Jones conjecture for CAT(0) groups \cite{Borel}. Then, the Borel conjecture was derived for a class of groups containing hyperbolic and CAT(0) groups \cite{Borel}.

Equivariant asymptotic dimension and transfer reducibility have been extensively studied in the last years, mostly as a tool to prove the Farrell--Jones conjecture{.
The scope of this research involves} $\mathrm{GL}_n(\Z)$ \cite{GLnZ} and other linear groups \cites{refinements, GLn}, virtually solvable groups \cite{virtuallySolvable}, CAT(0) groups \cites{flowForCAT0, CAT0}, and relatively hyperbolic groups \cite{coarseFlowSpaces}. Very recently, in \cite{long&thin}, a new construction of covers was proposed, which -- in particular -- provides improved bounds on equivariant asymptotic dimension of hyperbolic groups. There are more positive results regarding transfer reducibility {\cites{GLnZ, GLn, virtuallySolvable, flowForCAT0, CAT0}} than equivariant asymptotic dimension \cites{coversForHyperbolic, coarseFlowSpaces}, because its definition is formally less restrictive. However, it seems to be an open question whether the two notions are equivalent, cf. \cite{onProofs}*{Remark 3.15}.

All the known proofs \cites{long&thin,coarseFlowSpaces,coversForHyperbolic} showing finiteness of equivariant asymptotic dimension are complex and involve the notion of (coarse) flow space. Some elementary constructions, even in the simplest cases such as that of the free group, are unknown and desired, cf. \cite{onProofs}*{Remark 3.12}. We make a step in this direction, showing that it is enough to study the equivariant asymptotic dimension $\eqasdim G\times\dX$ for the boundary $\dX$. More precisely, we describe a method of extending coverings from the boundary and obtain the following result.

\begin{thmA}[Theorem \ref{thm:extendingFromBoundary}] Under appropriate assumptions, \vspace{-.1cm}
$$ \eqasdim G\times \bX \leq \eqasdim G\times \partial X + \dim X.$$ 
\end{thmA}

Another quantitative result is the following characterisation of the situation when the equivariant asymptotic dimension vanishes, which yields, as a corollary, a geometric characterisation of virtually cyclic groups.

\begin{thmB}[Theorem \ref{thm:zero-dim}] For a family of groups $\cF$, $\feqasdim G = 0$ if and only if $\cF$ contains a finite-index subgroup of $G$.
\end{thmB}

The notion of equivariant asymptotic dimension relates to some other concepts, most importantly to asymptotic dimension (but also to amenable actions). The similarity of definitions and quantitative relations are discussed in Subsection \ref{S:zeroDim}. 
The second part of the paper,
Section \ref{sec:characterisations},
is devoted to
providing a number of different characterisations of equivariant asymptotic dimension and transfer reducibility (Theorem \ref{thm:characterisations}). Interestingly, appropriate forms of characterisations invented originally for asymptotic dimension are still valid in the very general framework of homotopy actions (transfer reducible groups). In 
Proposition \ref{prop:equivariantMultiplicity}, we present two more characterisations and a different proof (not using metrisability) of their equivalences, assuming that we deal with ordinary (not homotopy) group actions.

In {\appendixname} \ref{equivariantTopological}, we strengthen the result of \cite{refinements} stating that for an equivariant covering one can find an equivariant refinement of dimension at most equal to the dimension of the space. In that sense, ``equivariant topological dimension'' is equal to the topological dimension. The theorem was originally formulated for finite groups and we generalise it to infinite groups {provided the action is assumed to be proper}. It is used in Subsection \ref{s:roleOfX} in the proof of Lemma \ref{lem:interiorCovering}, but can be read independently from the rest of the paper.

%
%

\section{Genuine group actions -- the vanishing theorem and extending coverings from the boundary}\label{s:first}

\subsection{Definition}

Let us start by fixing some notation. The metric neighbourhood of a subset $A$ of radius $r$ will be denoted by $B(A,r) = \bigcup_{x\in A}B(x,r)$ and $\oB(x,r)$ will denote a closed ball. When a group $G$ acts on a topological space $X$ (on a set $Y$), we will shortly say that $X$ ($Y$) is a $G$-space (a $G$-set).
Sometimes we will write ``for all $\alpha<\infty$...'' to denote ``for all ${\alpha\in(0,\infty)}$...'' in order to clarify that the following condition is trivial for ``small'' $\alpha$ and interesting for ``large'' ones.

Unless stated otherwise, we will assume that $G$ is a finitely generated group with a fixed word-length metric, and $\bX$ will denote a compact $G$-space.

\begin{defi}\label{defi:family}
A \emph{family $\mathcal F$ of subgroups} of a group $G$ is a set of subgroups closed under conjugation and taking subgroups.

A family $\cF$ is \emph{virtually closed} if for every $H\in\cF$ and $H\leq H'\leq G$ such that
$[H':H]<\infty$, also $H'\in \cF$. 
\end{defi}

Our considerations are general enough to hold for any family $\cF$ { as in Definition \ref{defi:family}}. However, in the context of the Farrell--Jones conjecture it is the (virtually closed) family of virtually cyclic subgroups, denoted $\vcyc$, that appears most naturally \cites{flowForCAT0,coversForHyperbolic,Borel,FarrellJonesForHyperbolic}.

\begin{defi}\label{Fsubset} Let $Y$ be a $G$-{set} and $\cF$ be a family of subgroups of~$G$. A subset $U\subseteq Y$ is called an \emph{$\cF$-subset} if:
\begin{enumerate}[label={(\alph*)}]
\item elements $gU$ of the orbit of $U$ are either equal or disjoint,
\item the stabiliser of $U$, $G_U = \{ g\in G \st gU=U \}$, is a member of $\cF$.
\end{enumerate}
\end{defi}

A cover that consists of $\cF$-subsets and is $G$-equivariant will be called an \emph{$\cF$-cover}. The name ``equivariant asymptotic dimension'' comes from the fact that the coverings in its definition are $\cF$-covers.

For a family of subsets $\U$ of set $Y$, by $\dim \U$ (the dimension of $\U$) we will denote the {value} $\sup_{y\in Y} | \{U\in \U \st y\in U \}| - 1$, where $|A|$ is the cardinality of $A$.

\begin{defi} Let $Y$ be any set and $\U$ be a covering of $G\times Y$. We say that $\alpha<\infty$ is \emph{a $G$-Lebesgue number} of $\U$, given that for each $(g,y)\in G\times Y$ there exists $U\in \U$ such that $\oB(g,\alpha)\times\{y\}\subseteq U$.
\end{defi}

The following definition originates in \cite{coversForHyperbolic}*{Theorem~1.1}, see also \cite{FarrellJonesForHyperbolic}*{Assumption~1.4} and \cite{coarseFlowSpaces}*{Definition 0.1}.

\begin{defi}\label{def:eqasdim}
\emph{{The} equivariant asymptotic dimension} of $G\times \bX$ with respect to family $\cF$, denoted by $\feqasdim G\times \bX$, is the smallest integer $n$ such that for every $\alpha < \infty$ there exists an open
$\cF$-cover $\U$ of $G \times \bX$ (with the diagonal $G$-action) satisfying:
\begin{enumerate}
\item $\dim(\U) \le n$,
\item $\alpha$ is a $G$-Lebesgue number of $\U$.
\end{enumerate}
If no such integer exists, we say that the dimension is infinite.
\end{defi}

When the family $\cF$ is irrelevant or clear from the context, we will skip it from notation. The coverings $\U=\U(\alpha)$ from the above definition will be called \emph{$\eqasdim$-coverings} and \emph{$\alpha$-$\eqasdim$-coverings}, in case the constant $\alpha$ is important.

\begin{longRem}
In \cite{coversForHyperbolic}*{Theorem 1.1}, $\eqasdim$-coverings were also required to be $G$-cofinite, but by compactness one can choose cofinite subcoverings from arbitrary coverings, so this requirement can be skipped.
\end{longRem}

One generalisation of equivariant asymptotic dimension (still sufficient for the Farrell--Jones conjecture) is transfer reducibility. It occurs in many flavours in the literature, but the main difference between it and equivariant asymptotic dimension is that in transfer reducibility one can choose a space $\bX$ depending on {a} parameter (for $\eqasdim$ this parameter is $\alpha$) and instead of a genuine group action a ``homotopy action'' is considered. Very roughly, in homotopy action the action of $gh$ is equal to the composition of actions of $g$ and $h$ only up to homotopy. We study this notion in Section \ref{sec:characterisations}, see in particular Definition \ref{def:homotopyAction} and Remark \ref{def:transferReducible}.

\begin{longRem}
Note that if we have a $G$-equivariant map $p\colon\myol Y\to \bX$, then $\eqasdim$-coverings of $G\times \bX$ can be pulled back to $\eqasdim$-coverings of $G\times \myol Y$. Hence, the minimal possible value of $\eqasdim G\times \bX$ for $\bX$ compact and Hausdorff is acquired for $\bX=\beta G$ -- we will sometimes call it \emph{the equivariant asymptotic dimension of $G$}.
\end{longRem}

It is not enough to restrict to $\bX=\beta G$, though, since in applications conditions similar to the following are utilised, cf. \cite{FarrellJonesForHyperbolic}*{Theorem 1.1 and Assumption 1.2}. 
\begin{itemize}
\item $\bX$ is a metrisable compactification of its $G$-invariant subset $X$,
\item $X$ is a realisation of an abstract simplicial complex,
\item $\bX$ is contractible,
\item (weak Z-set condition) there exists a homotopy $H\colon \bX\times[0,1] \to \bX$, such that $H_0=\id_\bX$ and $H_t(\bX)\subseteq X$ for every $t>0$.
\end{itemize}
However, we do not adopt any of these conditions as a convention.

In the above context, considerations may become less complicated if one constructs coverings of $G\times \dX$ rather than the whole $G\times \bX$ (where $\dX=\bX\setminus X$). The fact that the latter can be reconstructed from the former is the content of Theorem \ref{thm:extendingFromBoundary}.

A natural setting to have in mind is when the space $X$ admits a geometric action of $G$ (for example, it is a Rips complex of the group) and $\dX$ is the Gromov boundary of $G$, cf. \cite{coversForHyperbolic}.

%
%

\subsection{\texorpdfstring{The role of $\boldsymbol\bX$}{The role of cl(X)}}\label{s:roleOfX}

Since the applications of equivariant asymptotic dimension concern the group $G$, it is natural to ask what the role of $\bX$ is. What happens if we just take $\bX=\{*\}$ or, on the other extreme, drop the compactness assumption. A special case of the latter is the question whether we can drop the requirement that $\eqasdim$-coverings are open, which corresponds to equipping $\bX$ with the discrete topology. Another question is whether a decomposition of $\bX$ into invariant subspaces can be used to simplify the problem of finding $\eqasdim$-coverings.

It turns out that compactness of $\bX$ is crucial to the notion of $\eqasdim$.

\begin{rem} If the compactness assumption for $\bX$ in Definition \ref{def:eqasdim} is skipped, then for $X=G$ we have $\eqasdim G\times X = 0$.
\end{rem}

\begin{proof}
A good $\eqasdim$-covering for $G\times X$ is $\U = \{ G \times \{x\} \st x\in X \}$, which is clearly an open $\mathcal{T}$-cover of dimension $0$ with the infinite $G$-Lebesgue number, where $\mathcal T$ is the singleton family of the trivial subgroup of $G$.
\end{proof}

Clearly, the same construction of $\eqasdim$-coverings works for any discrete $X$ provided that point stabilisers belong to $\cF$. The above proof exemplifies a more general approach indicated in \cite{coversForHyperbolic}. While $\eqasdim$-coverings must be $\alpha$-large in the $G$-coordinate, making them small in the $\bX$-coordinate may be helpful in obtaining the properties desired in Definition \ref{Fsubset}. The following lemma generalises the above remark and covers a wide range of examples (e.g., the spaces considered in \cite{coversForHyperbolic}). Below, a $G$-simplicial complex such that all the stabilisers of simplices belong to $\cF$ is called an \emph{$\cF$-simplicial complex}.

\begin{lem}\label{lem:interiorCovering}
Assume that a finitely generated group $G$ acts on a topological space $X$. There is an $\feqasdim$-covering $\U_\infty$ (with $\alpha = \infty$) of the space $G\times X$ under any of the following conditions:
\begin{enumerate}[label={(\alph*)}]
\item\label{interiorCoveringA} $X$ is a finite-dimensional $\cF$-simplicial complex
(the same is true for CW-complexes);
\item\label{interiorCoveringB} $X$ is regular, the $G$-action is proper, isotropy groups belong to $\cF$, and either of the following conditions holds:
\begin{enumerate}[label={(\roman*)}]
\item\label{interiorCoveringBi} the $G$-action is cocompact;
\item\label{interiorCoveringBii} $X$ is of finite covering dimension and admits a $G$-invariant metric.
\end{enumerate}
\end{enumerate}
{In cases \ref{interiorCoveringA} and \ref{interiorCoveringBii}, we have $\dim \U_\infty\leq \dim X$.}
\end{lem}
\begin{proof}
Ad \ref{interiorCoveringA}. For each open simplex $\Delta^\mathrm{o}$ of $X$, one can construct (using a barycentric subdivision, compare the proof of implication $(\ref{cond:epsGMap}) \implies (\ref{cond:disjointEqasdim})$ in Theorem \ref{thm:characterisations} or \cite{refinements}*{Lemma 3.4} for more details) a neighbourhood $N(\Delta^\mathrm{o})$ such that neighbourhoods of simplices of the same dimension are disjoint and the family of such neighbourhoods is equivariant. Thus, the stabiliser of $N(\Delta^\mathrm{o})$ is equal to the stabiliser of $\Delta^\mathrm{o}$ and hence belongs to $\cF$. Putting $\U = \{G\times N(\Delta^\mathrm{o}) \st \Delta\in X \}$ finishes the proof, because each point $x$ of $X$ belongs to a neighbourhood of at most one simplex of each dimension.

Ad \ref{interiorCoveringB}. For each $x\in X$ we will construct its neighbourhood $U_x$ being an $\cF$-subset. By properness of the action (and $T_1$-property), we can find a neighbourhood $U_x^0$ such that the set $RS_x = \{ g \st gU_x^0 \cap U_x^0 \neq \emptyset \}$ is finite and such that $U_x^0$ is disjoint with the completion $C_x = Gx\setminus \{x\}$ of $x$ in its orbit $Gx$.

Then, using regularity of $X$, we choose a smaller neighbourhood $U_x^1$, such that its closure $\myol{U_x^1}$ is contained in $U_x^0$ -- in particular it is disjoint with $C_x$. But we have the equivalence 
$$\forall_{g: gx\neq x} gx\not\in  \myol{U_x^1} \iff
  \forall_{g: gx\neq x} x \not\in g\myol{U_x^1},$$
so the set $U_x^2 = U_x^1 \setminus \bigcup_{g: gx\neq x} g\myol{U_x^1}$ contains $x$. It is open, as the sum can be taken over the finite set $RS_x$ without affecting the difference. What we achieved is emptiness of the intersection $U_x^2 \cap gU_x^2 \subseteq \left(U_x^1 \setminus gU_x^1\right)\cap gU_x^1 = \emptyset$ for $gx\neq x$.

To handle the case $gx=x$, we do the last tweak setting $U_x = \bigcap_{g:gx=x}gU_x^2$. The intersection is finite (as the stabiliser of $x$ is a subset of $RS_x$), so we have just obtained a neighbourhood of $x$ with the stabiliser equal to the stabiliser of $x$, and conclude that $U_x$ is an $\cF$-subset.

We still need to bound the dimension of the covering. Provided that $X$ is finite-dimensional and the action is isometric, we can use Proposition \ref{prop:refinements} to find an equivariant refinement $\U_X$ of the covering $\{gU_x\st g\in G,\ x\in X\}$ with dimension at most $\dim X$.

Otherwise, we can assume that the action is cocompact. Since the quotient map $X\overset{q}\to  {X}/{G}$ is open, $\left\{q(U_x)\right\}_{x\in X}$ is an open covering of a compact set. Consequently, there is a finite family $x_0,\ldots, x_n$ such that $\left\{q(U_{x_i})\right\}_{0\leq i \leq n}$ covers $X/G$ and thus $\U_X=\left\{g U_{x_i} \st g\in G,\ 0\leq i \leq n \right\}$ covers $X$. Clearly, the dimension of $\U_X$ is at most $n$.

Finally, the family $\U_\infty = \{ G\times U \st U \in \U_X \}$ is an $\alpha$-$\eqasdim$-covering of $G\times X$ for any $\alpha\leq \infty$.
\end{proof}

We would like to mention that actually $G$-invariant coverings of $X$ (rather than of $G\times X$) were constructed in the above proof and that it relied mainly on topological properties of $X$ (not on the geometry of $G$).

Assume now that $\bX$ is a compactification of $X$ and recall that we denote $\dX=\bX\setminus X$. An $\eqasdim$-covering $\U$ of $G\times \bX$ breaks up into two invariant parts:
$$\U^{\mathrm{o}} = \{ U\in \U \st U\cap G\times \dX = \emptyset \},$$
$$\U^\partial = \{ U\in \U \st U\cap G\times \dX \neq \emptyset \}.$$
Conversely, if we are given two open $\cF$-families $\U^{\mathrm{o}}$, $\U^\partial$ of subsets of $G\times \bX$, which -- after restriction to $G\times X$ and $G\times \dX$ respectively -- have $G$-Lebesgue numbers $\alpha$, then the family $\U^{\mathrm{o}} \cup \U^\partial$ is an $\cF$-cover of $G\times \bX$ with $G$-Lebesgue number $\alpha$ and dimension at most $\dim \U^{\mathrm{o}} + \dim \U^\partial + 1$.

Hence, if the assumptions of Lemma \ref{lem:interiorCovering} are satisfied, we can always assume (at the expense of possible increase in the bound on the dimension) that $\eqasdim$-coverings $\U$ of $G\times \bX$ satisfy $\U^{\mathrm{o}}=\U_\infty$ and thus the only relevant part of $\U$ is $\U^\partial$. In other words, it is enough to deal with \emph{a neighbourhood of} the boundary to obtain a covering of $G\times \bX$.

Even more is true. An open $\cF$-cover of $G\times \dX$ can be extended to a family $\U^\partial$ of the same dimension that is open in $G\times \bX$. Thus, one can indeed restrict their attention to the boundary itself.

\begin{thm}\label{thm:extendingFromBoundary}
If $\bX$ is a metrisable compactification of $X$ and any of the assumptions of Lemma \ref{lem:interiorCovering} hold, 
then $$\eqasdim G\times\bX \leq \eqasdim G\times \dX + \dim \U_\infty + 1,$$
where $\U_\infty$ is the covering from Lemma \ref{lem:interiorCovering}.
\end{thm}

\begin{proof} For a given $\alpha<\infty$, we will define an $\alpha\heqasdim$-covering $\U$ of $G\times \bX$ as the sum of the covering $\U_\infty$ and a covering $\U^\partial$ constructed from an $\alpha$-$\eqasdim$-covering $\V$ of $G \times \dX$.

For $V\subseteq G\times\dX$ let $V_g=V \cap \left(\{g\}\times \dX\right)$ and let $\V_1 = \{V_1 \st V\in \V \}$. We will describe a dimension-preserving method of enlarging sets $Y\in\V_1$ to open subsets $W(Y)$ of ${\{1\}\times\bX\simeq}\bX$. It will satisfy:
\begin{eqnarray}\label{eq:extending;trace}
&
Y\neq Y'\implies W(Y)\neq W(Y');
&\\
&\label{eq:extending;emptyIntersection}
\displaystyle\bigcap_{1\leq i\leq n}W(Y^i) =  W\left(\bigcap_{1\leq i\leq n}Y^i\right).
&
\end{eqnarray}

Fix a metric $d$ inducing the topology of $\bX$. As $\dim \V_1 \leq \dim \V$, every $x\in \dX$ belongs to a finite number of elements $Y^1_x, \ldots, Y^k_x$ of $\V_1$. Let $\eps(x)>0$ be such that $B(x,\eps(x)) \cap \dX \subseteq \bigcap_j Y^j_x$. For any $Y\subseteq \dX$ we define $W(Y) = \bigcup_{x\in Y} B(x,\eps(x)/2)$.

Condition (\ref{eq:extending;trace}) is clear from the fact that $W(Y)\cap \dX = Y$. We will show condition (\ref{eq:extending;emptyIntersection}) by induction. Let us denote $Y^{(n)} = \bigcap_{1\leq i\leq n} Y^i$. The base is trivial, so we assume $n>1$:
\begin{align*}
\bigcap_{1\leq i\leq n}W\left(Y^i\right) =& \bigcap_{1\leq i\leq n-1}W\left(Y^i\right)\cap W\left(Y^n\right) = W\left(Y^{(n-1)}\right) \cap W\left(Y^n\right) \\
=& \bigg(W\left(Y^{(n-1)}\setminus Y^n\right) \cup W\left(Y^{(n-1)}\cap Y^n\right)\bigg) \\
&\cap
\bigg(W\left(Y^n \setminus Y^{(n-1)}\right) \cup W\left(Y^n \cap Y^{(n-1)}\right)\bigg)\\
=& \left(W\left(Y^{(n-1)}\setminus Y^n\right) \cap W\left(Y^n \setminus Y^{(n-1)}\right)\right) \cup W\left(Y^{(n)}\right)
\end{align*}
We claim that the first summand of the right-hand side is empty. Suppose some $z$ belongs to it. Then, there must be $x\in Y^{(n-1)}\setminus Y^n$ and $y\in Y^n \setminus Y^{(n-1)}$ such that $z\in B(x, \eps(x)/2) \cap B(y, \eps(y)/2)$. Thus, $d(x,y) < \max(\eps(x), \eps(y))$. Hence, either $x\in B(y, \eps(y)) \cap \dX \subseteq \bigcap_j Y^j_y \subseteq Y^n$ (contradicting $x\in Y^{(n-1)}\setminus Y^n$), or $y\in B(x, \eps(x)) \cap \dX \subseteq \bigcap_j Y^j_x \subseteq Y^{(n-1)}$ (contradicting $y\in Y^n \setminus Y^{(n-1)}$).

Now, for $V\in \V$ with a decomposition $V = \bigcup_g \{g\} \times V_g$ let us define 
$U(V) = \bigcup_g \{g\} \times gW(g^{-1}V_g)$ (note that $g^{-1}V_g = {\left(g^{-1}V\right)}_1\in \V_1$). We have the equality $hV = \bigcup_g \{g\} \times hV_{h^{-1}g}$, so one gets:
$$U(hV) = \bigcup_{g\in G} \{g\} \times gW(g^{-1}hV_{h^{-1}g}) = \bigcup_{k\in G} \{hk\} \times hkW(k^{-1}V_k) = h U(V);$$
i.e., the obtained family is equivariant.

From (\ref{eq:extending;trace}) and the equivariance it follows that the stabiliser of $U(V)$ is equal to the stabiliser of $V$, hence it belongs to $\cF$. Condition (\ref{eq:extending;emptyIntersection}) (for $n=2$) guarantees that different translates $U(V)$, $gU(V)=U(gV)$ are disjoint and (for arbitrary $n$) it assures that the dimension of $\U^\partial = \{ U(V) \st V\in \V\}$ is equal to the dimension of $\V$. Putting $\U=\U^\partial \cup \U_\infty$ finishes the proof.
\end{proof}

\begin{longEx} The crucial part of the proof is the definition of sets $W(Y)$. For $X=\Cay(\F_2,\{a,b\})$ and its Gromov boundary as $\dX$, it is enough to define $W(Y)\cap X$ as all those points $x\in X$ such that the endpoint of any geodesic ray from $1$ via $x$ ends in $Y$ (geometrically: we take ``cones'' over the boundary).
\end{longEx}

Group $G$ can be embedded in $G\times \bX$ in various ways yielding different pullbacks of $\U$. Assume for example that the conclusion of the \v{S}varc--Milnor lemma is true, that is, the orbit map $G\owns g\mapsto gx_0 \in X$ is a quasi-isometry for some metric on $X$, and consider the pullback of $\U$ in $G$ via the map $g\mapsto (1,gx_0)$. Inverse images of $U\in \U_\infty$ will be uniformly bounded. On the other hand, we expect the inverse image of $U\in \U^\partial$ to be unbounded, as it contains neighbourhoods of ``points at infinity'', cf.\ condition \eqref{asymptoticity} in the Remarks subsection. Thus, some elements of the covering are small, independently of $\alpha$, while others are unbounded. It differs from $\asdim$-coverings (see Definition \ref{def:asdim} below), where elements of the covering grow with $\alpha$, but uniform boundedness is preserved at each step.

Nonetheless, equivariant and classical asymptotic dimension are related and we discuss it in the next subsection.

%
%

\subsection{Asymptotic dimension and the vanishing theorem}\label{S:zeroDim}

A natural question coming to mind is how equivariant asymptotic dimension is related to asymptotic dimension. Let us recall the definition.

\begin{defi}\label{def:asdim}
\emph{The asymptotic dimension} of a metric space $G$ is the smallest integer $n$ such that for all $\alpha<\infty$ there is an open covering $\U$ of $G$ such that:
\begin{enumerate}
\item $\dim(\U) \le n$,
\item for each $g \in G$ there exists $U \in \U$ such that
      $\oB(g, \alpha) \subseteq U$,
\item $\sup_{U\in\U}\diam(U)<\infty$ (uniform boundedness).
\end{enumerate}
\end{defi}

We can see that the first two conditions in the definition of $\asdim$ are analogues of the conditions for $\eqasdim$. Such a similarity occurs also for various characterisations of $\asdim$, compare the following Theorem \ref{thm:asdimCharacterisations} (see \cite{Bedlewo}*{Theorem 1} or \cite{Roe}*{Theorem 9.9}) characterising $\asdim$ with Theorem \ref{thm:characterisations} and Proposition \ref{prop:equivariantMultiplicity} characterising $\eqasdim$.

\begin{thm}\label{thm:asdimCharacterisations} Let $X$ be a metric space. The following conditions
are equivalent.
\begin{enumerate}
    \item\label{cond:asdim} $\asdim X\leq n$;
    \item\label{cond:disjoint} for every $r<\infty$ there exist uniformly bounded, $r$-disjoint families $(\U^i)$ for $0\leq i \leq n$ of subsets of $X$ such that $\bigcup_i\U^i$ covers $X$.
    \item\label{cond:epsToComplex} for every $\eps>0$ there is a uniformly cobounded, $\eps$-Lipschitz map $\phi\colon X\to K$ to a simplicial complex of dimension $n$.
    \item\label{cond:multiplicity} for every $d<\infty$ there exists a uniformly bounded
    cover $\V$ of $X$ with $d$-multiplicity at most $n+1$;
\end{enumerate}
\end{thm}

In the above theorem, a family of subsets is \emph{$r$-disjoint} if the distance of any two of its members is at least $r$; a map to a simplicial complex is \emph{uniformly cobounded} if there is a bound on diameter of inverse images of stars; a simplicial complex $K$ is viewed as a subset of $\ell_1(V(K))$, where $V(K)$ is the set of vertices of $K$; and \emph{$d$-multiplicity} of a covering means the maximal number of its elements intersecting a $d$-ball.

We would like to mention that if we relax condition (\ref{cond:epsToComplex}) so that $\phi$ is a map into the sphere in $\ell_1$ instead of an $n$-dimensional complex, then it becomes (in the case of bounded geometry metric spaces) equivalent to property A (\cite{someNotesOnA}*{Theorem 1.2.4 (6)}, see also \cite{propertyAAndAsdim}). The equivalence is established by replacing $\eps$-Lipschitz maps $x\mapsto \phi(x)$ into the unit sphere of $\ell_1(V(K))$ by maps $x\mapsto A_x$ into $\left(\bigoplus_{v} \N\right)\setminus\{0\}$, where $v\in V(K)$, such that for $d(x,y)\leq R$ we have $\| A_x - A_y \| / \|\min(A_x, A_y) \| < \eps $ (and \emph{vice versa}). A similar transition would give one more characterisation also in the equivariant case.

Guentner, Willet, and Yu \cite{G-W-Yu} show that
$\mathcal{F}in\heqasdim G\times \beta G = \asdim G$, where $\mathcal{F}in$ is the family of finite subgroups. Clearly, the equivariant asymptotic dimension decreases when $\cF$ increases, so $\feqasdim G\times \beta G \leq \asdim G$ for any $\cF\supseteq \mathcal Fin$.

On the other hand, Willett and Yu observed that the appropriate version of the argument from \cite{BrownOzawa}*{Proposition 5.2.1} gives the inequality
$$1 + \asdim G\leq (1 + \cF\heqasdim G\times \bX) \cdot (1 + \sup_{F\in \cF}\asdim F);$$
in particular, finite $\vcyc\heqasdim$ implies finite asymptotic dimension of $G$. The proof uses the language of \emph{amenable actions} (we recall the definition in the discussion after Theorem \ref{thm:characterisations}).

The definition of equivariant asymptotic dimension involves a family of groups $\cF$, for example $\vcyc$, which causes the two notions -- classical and equivariant asymptotic dimension -- to disagree even in the simplest cases. In particular, the second factor on the right-hand side of the above inequality is necessary, because it is \emph{not} true that $\asdim G \leq \feqasdim G$:

\begin{longEx}\label{CyclicToZero}
For $G=\Z$ and, say, $\bX = [-\infty,+\infty]$ with the action by translations, the one-element covering $\{ G\times \bX \}$ is an $\alpha\heqasdim$-covering for any $\alpha<\infty$ and $\cF=\vcyc$. Hence, $\vcyc\heqasdim \Z = 0$, while $\asdim \Z = 1$.
\end{longEx}

Apparently, $\eqasdim$ is a more subtle (or at least less understood) notion than $\asdim$ -- for several years the only class of groups that were known to be of finite equivariant asymptotic dimension with a reasonable $\bX$ was the family of hyperbolic groups \cite{coversForHyperbolic}, and the fact that they also have finite asymptotic dimension is classical and has a short proof \cite{hyperbolicAsdim}. (Now, we have similar results for relatively hyperbolic groups \cite{coarseFlowSpaces} and mapping class groups \cite{BartelsBestvina}.) On the other hand, we have no examples of finite-$\asdim$ groups which are known to have infinite equivariant asymptotic dimension if we put some restrictions on $\bX$.

The difficulty with {proving finiteness of} $\eqasdim$ arise{s} (see the Remarks subsection) already in the case of the simplest non-hyperbolic group, namely $\Z^2$, which can be immediately proven to be of asymptotic dimension~$2$.

It turns out that Example \ref{CyclicToZero} can be generalised to give a complete characterisation of groups with vanishing equivariant asymptotic dimension.

\begin{thm}\label{thm:zero-dim} Equivariant asymptotic dimension of $G$ vanishes if and only if $\cF$ contains a finite-index subgroup of $G$.
\end{thm}

\begin{proof}
For the ``if'' part, let $H\in \cF$ be a finite-index subgroup of $G$. Let $\bX$ be the quotient $G/H$ and $\U = \{G\times \{x\} \st x\in \bX \}$. It is a $G$-invariant, disjoint and open covering with respect to the discrete topology on $\bX$, and stabilisers of elements of $\U$ are conjugates of $H$, hence they belong to $\cF$.

For the converse, assume that there is an $\cF$-cover $\U$ of $G\times \bX$ of dimension $0$ (that is, disjoint) and of $G$-Lebesgue number $\alpha \geq 1$. Take $U\in \U$ and $(g,x)\in U$. Then there exists $U'\in \U$ such that $\oB(g,\alpha)\times\{x\} \subseteq U'$ -- but then $U\cap U'\neq \emptyset$, so $U=U'$. Thus, we showed that $(g,x)\in U$ implies $\oB(g,\alpha)\times\{x\} \subseteq U$ -- hence $G\times \{x\} \subseteq U$ and we conclude that $U = G\times U_\bX$ for an open set $U_\bX\subseteq \bX$.

Consider now the sum $W = \bigcup GU_\bX$ of the orbit $GU_\bX$. We claim that $W$ is closed. Indeed, for $y\in \bX\setminus W$, there is $U'=G\times U'_\bX \in \U$ such that $y\in U'_\bX$, and disjointness of $\U$ implies $U'_\bX\cap W = \emptyset$.

So $W$ is a compact subset of $\bX$ covered by the disjoint family $GU_\bX$, meaning that the family must be finite. Thus, the orbit of $U_\bX$ is finite and the same is true for the orbit of $U$. Summing up, the stabiliser of $U$ belongs to $\cF$ and is of finite index in $G$.
\end{proof}

\begin{cor} $\vcyc\heqasdim G = 0$ if and only if $G$ is virtually cyclic.
\end{cor}

\begin{longEx} In particular, $\eqasdim$ is not a function of the asymptotic dimension of a group and/or the topological dimension of its boundary, as $\asdim \Z = \asdim \mathbb F_n = 1$ and $\dim \partial \Z = \dim \partial \mathbb F_n = 0,$ but for $n>1$: $$\vcyc\heqasdim \Z = 0 < \vcyc\heqasdim \F_n.$$ 
\end{longEx}

Moreover, equivariant asymptotic dimension does \emph{not} satisfy a logarithmic inequality holding for other notions of dimension ($\dim G\times H \leq \dim G + \dim H$), as $\eqasdim \Z^n > 0 = \eqasdim \Z$. In fact, it seems to be an open problem whether the product of groups of finite $\eqasdim$ has finite $\eqasdim$.

\subsubsection*{Remarks}\label{remarks:abelian}

Let us now consider the following situation. {Assume that $G$ acts geometrically on $(X,d)$ for a suitable proper and geodesic metric $d$. Then, any orbit} map $G\owns g\overset{j}{\mapsto} gx_0 \in X$ is a quasi-isometry{. Assume further that} if a sequence $(x_n)$ of points of $X$ converges to $x\in\dX$ and $(y_n)\in X$ is asymptotic to $(x_n)$, then $(y_n)$ also converges to $x$:
\begin{equation}\label{asymptoticity}
\sup_n d(x_n, y_n)<\infty \implies \lim_n y_n = x
\end{equation}
(which holds in particular for the Gromov boundary of a hyperbolic space).
Then, any finitely generated abelian subgroup $H\leq G$ has to belong to family $\cF$.

Indeed, let $\{z_1,\ldots z_k \}$ be a generating set of $H$, and let a sequence $(h_n)\in H$ and a point $x_0\in X$ be such that $x_n = h_nx_0$ converges to some $x\in \dX$. Then for any $i\in\{1,\ldots,k\}$ and a sequence $y_n= h_nz_ix_0$ we have $\sup_n d(x_n, y_n)<\infty$ (as $d_G(h_n, h_nz_i) = d_G(1, z_i)$ and $j$ is a quasi-isometry), and thus $z_i x = \lim z_i h_n x_0 = \lim h_n z_i x_0 = x$; i.e., $H$ stabilises $x$.

But finitely generated subgroups of isotropy groups of $\bX$ belong to $\cF$:
let $\alpha$ be large enough for $\oB(1,\alpha)$ to contain $\{z_1,\ldots z_k\}$ and let $U\in\U(\alpha)$ contain $\oB(1,\alpha) \times \{x\}$. Then $z_i \left(\oB(1,\alpha) \times \{x\}\right) = \oB(z_i,\alpha) \times \{x\}$ intersects nontrivially with $\oB(1,\alpha) \times \{x\}$ and thus $z_iU \cap U\neq \emptyset$, so $z_i$ must stabilise $U$ and thus $H$ is a subgroup of the stabiliser of $U$, which belongs to $\cF$ by the definition of $\feqasdim$-coverings.
This reasoning also shows that space $\bX$ in the definition of $\eqasdim$ is necessary; i.e., there are no $\alpha$-$\eqasdim$-coverings of $G = G\times \{*\}$ (unless $\alpha<1$ or $G\in\cF$).

The above suggests that commutativity (or existence of large abelian subgroups) may be an obstacle for (proving) finiteness of $\eqasdim$. Such a proof (if we assume $X\simeq G$) would require a compactification violating very natural condition \eqref{asymptoticity}. The condition holds 
for compactifications of CAT(0) groups used in \cite{flowForCAT0}, and thus {Bartels and L\"uck} used suitable subspaces of the compactification and showed transfer reducibility (not finiteness of $\eqasdim$).

However, if we already have a \emph{free} action with $G$ nilpotent and $\bX$ metrisable, we can use 
the following estimate
from
\cite{SWZ} (the paper uses the name \emph{amenability dimension}):
$$1 + \mathcal T\heqasdim G\times \bX\leq 3^{\ell(G)}\cdot (1+\dim \bX),$$
where $\mathcal T$ is the singleton family of the trivial subgroup of $G$, $\ell(G)$ is the Hirsch length and $\dim \bX$ is the Lebesgue covering dimension of $\bX$.

%
%

\section{Characterisations of equivariant asymptotic dimension}
\label{sec:characterisations}

The aim of this section is to provide a number of equivalent characterisations of equivariant asymptotic dimension. We will state our theorem in a generality broader than in the previous section to handle the notion of transfer reducible groups that are defined in terms of homotopy group actions.

%
%

\subsection{Homotopy actions}

Consider the map $\rho(g,x)=(g,g^{-1}x)$. It is a $G$-equivariant homeomorphism from $G\times \bX$ with the diagonal action onto $G\times \bX$ with the action on the first coordinate by left multiplication (we will call this the action by translations). The condition $\oB(g,\alpha)\times\{x\}\subseteq U$, is equivalent to $DB^{\alpha}(\rho(g,x)) \defeq \rho(\oB(g,\alpha)\times\{x\}) \subseteq \rho(U)$, where $DB$ stands for ``diagonal (closed) ball'' and can be described as follows:
$$DB^\alpha(g,x) = \{(gh, h^{-1}x) \st |h|\leq\alpha \}.$$
Hence, we could define equivariant asymptotic dimension using diagonal balls and the action by translations on $G\times \bX$.

Note that in this reformulation the action on $\bX$ is used only to define $DB$s. Thus, {if we are} able to generalise the definition of $DB$, we could ease the requirement that $G$ acts on $\bX$.

\begin{defi}[{\cite{flowForCAT0}*{Definition 0.1}}]\label{def:homotopyAction} Let $\bX$ be a compact metric space, and $S$ a finite and symmetric subset of a group $G$ containing the neutral element $1$.

\begin{enumerate}[label={(\roman*)}]
\item
A \emph{homotopy $S$-action $(\varphi,H)$ on $\bX$} consists of continuous maps $\varphi_g \colon \bX \to \bX$ for $g \in S$ and homotopies $H_{g,h}^t \colon \bX \to \bX$ for $g,h \in S$ with $gh \in S$ and $t\in [0,1]$ such that $H_{g,h}^0 = \varphi_g \circ \varphi_h$ and $H_{g,h}^1 = \varphi_{gh}$. Moreover, we require $H_{1,1}^t = \varphi_1 = \id_\bX$ for all $t \in [0,1]$;

\item 
Let $(\varphi,H)$ be a homotopy $S$-action on $\bX$. For $g \in S$ let $F_g=F_g(\varphi,H)$ be the set of all maps $H_{r,s}^t$, where $rs = g$.

For $(g,x) \in G \times \bX$, let $DB^1_{\varphi,H}(g,x)$ be the subset of $G \times \bX$ consisting of all $(gs,y)\in G\times X$ such that $y=f_{s^{-1}}(x)$ or $x=f_{s}(y)$,
where $s\in S$, $f_{s^{-1}}\in F_{s^{-1}}$ and $f_{s}\in F_{s}$.
For $A\subseteq G\times \bX$ and $n \in \N$ we put $DB^1_{\varphi,H}(A)=\bigcup_{(g,x)\in A} DB^1_{\varphi,H}(g,x)$ and inductively $DB^{n+1}_{\varphi,H}(A) = DB^1_{\varphi,H}\left(DB^n_{\varphi,H}(A)\right)$

\item Let $(\varphi,H)$ be a homotopy $S$-action on $\bX$, $\U$ be an open cover of $G \times \bX$, and $n\in \N$. We say that $\U$ is \emph{$n$-long with respect to $(\varphi,H)$} if for every $(g,x) \in G \times \bX$ there is $U \in \U$ containing $DB^{n}_{\varphi,H}(g,x)$.
\end{enumerate}
\end{defi}

Note that -- due to the fact that $1\in S$ and $\id_\bX\in F_1(\varphi,H)$ -- we have $A\subseteq DB^1_{\varphi,H}(A)$, so $n\mapsto DB^n_{\varphi,H}(A)$ is ``increasing''.

\begin{lem}\label{lem:controllingDB} Let $A$ be a compact subset of $G\times\bX$. Then $DB^{n}_{\varphi,H}(A)$ is also compact. Moreover, for every $\eps>0$ there exists $\delta>0$ such that $DB^n_{\varphi,H}(B(A,\delta)) \subseteq B\left(DB^n_{\varphi,H}(A),\ \eps\right)$, where the neighbourhoods are taken with respect to the product metric on $G\times\bX$.
\end{lem}
\begin{proof} Since $DB^n_{\varphi,H}$ is the $n$-th power of $DB^1_{\varphi,H}$ (viewed as a function from the power set of $G\times \bX$ to itself), it is enough to restrict to the case of $DB \defeq DB^1_{\varphi,H}$. Moreover, we can restrict to the case when $A=\{h\}\times Y$, for some closed $Y\subseteq \bX$.

Observe that $DB(\{h\}\times Y)$ is the union of two sets $I$ and $II$:
\begin{align*}
I = \bigcup_{g\in S} \bigcup_{\substack{r,s\in S \\ rs=g^{-1}}} \bigcup_{\substack{t\in[0,1]}} &\{hg\} \times H_{r,s}^t(Y),\\
II = \bigcup_{g\in S} \bigcup_{\substack{r,s\in S \\ \phantom{..}rs=g\phantom{..}}} \bigcup_{\substack{t\in[0,1]}} &\{hg\} \times {(H_{r,s}^t)}^{-1}(Y).
\end{align*}

Let us denote the map $(x,t)\mapsto H^t_{r,s}(x)$ by $H_{r,s}$. Instead of $\bigcup_{\substack{t\in[0,1]}} H_{r,s}^t(Y)$ one can write $H_{r,s}(Y\times [0,1])$ and similarly $\bigcup_{\substack{t\in[0,1]}} {(H_{r,s}^t)}^{-1}(Y) = \pi_\bX\left({H_{r,s}}^{-1}(Y)\right)$, where $\pi_\bX\colon\bX\times[0,1]\to\bX$ is the projection. Consequently:
\begin{align*}
I = \bigcup_{g\in S} \bigcup_{\substack{r,s\in S \\ rs=g^{-1}}} &\{hg\} \times H_{r,s}(Y\times[0,1]),\\
II = \bigcup_{g\in S} \bigcup_{\substack{r,s\in S \\ \phantom{..}rs=g\phantom{..}}} &\{hg\} \times \pi_\bX\left({H_{r,s}}^{-1}(Y)\right).
\end{align*}
The above sums are finite, so the obtained set is compact, because images and inverse images of compact sets are compact as long as all the spaces considered are compact.

Maps $H_{r,s}$ are uniformly continuous, so the ``moreover'' part is clear for $I$, and in order to obtain it for $II$ it suffices to prove the following (because $\pi_\bX$ is also uniformly continuous).

\begin{claim}  Let $H\colon Z\to Z'$ be a continuous map between compact metric spaces. For every compact subset $A\subseteq Z'$ and every $\eps>0$ there exists $\delta>0$ such that $H^{-1}(B(A,\delta)) \subseteq B\left(H^{-1}(A),\ \eps\right)$.
\end{claim}

Suppose the contrary, that for some $\eps>0$ there is a sequence of $z_n'$ approaching $A$ such that there are $z_n\in H^{-1}(z_n')$ at least $\eps$-distant from $H^{-1}(A)$. By passing to a subsequence, we can assume that $z_n$ converge to some $z_0\notin B(H^{-1}(A), \ \eps)$. However, by continuity, $H(z_0) \in A$, which yields a contradiction. 
\end{proof}

%
%

\subsection{The characterisations}

\begin{defi} Let $Y$ be a $G$-set. Its subset is called an \emph{almost $\cF$-subset} if its stabiliser belongs to $\cF$. An almost $\cF$-cover is a covering consisting of almost $\cF$-subsets and closed under the induced action of $G$.
\end{defi}

That is, what distinguishes an almost $\cF$-subset $U$ from an $\cF$-subset is that it may happen that $U\neq gU$, but still $U\cap gU\neq \emptyset$.

Conditions (\ref{cond:eqasdim}), (\ref{cond:disjointEqasdim}), and (\ref{cond:epsGMap}) below correspond to conditions (\ref{cond:asdim}), (\ref{cond:disjoint}), and (\ref{cond:epsToComplex}) in Theorem \ref{thm:asdimCharacterisations}. Condition (\ref{cond:almostEqasdim}), a version of (\ref{cond:eqasdim}), is introduced to relate to the ``almost'' versions of transfer reducibility present in the literature, \cite{GLnZ}. Similar to condition (\ref{cond:epsGMap}), condition (\ref{cond:amenableAction}) comes from \cite{onProofs}*{Theorem A} and resembles the definition of an amenable action.

Maps from condition (\ref{cond:epsGMap}) yield functors crucial for the proofs of the Farrell--Jones and Borel conjectures, compare \cite{FarrellJonesForHyperbolic}*{Section 4}, \cite{CAT0}*{Proposition 3.6 and Section 5} and \cite{Borel}*{Proposition 3.9 and Section 11}.

\begin{thm}\label{thm:characterisations}
Let $n\in\N$ and $S$ be a finite symmetric subset of $G$ containing the identity element $1$. Below, we require each $\bX$ to be a compact metrisable space and $(\varphi,H)$ to be a homotopy $S$-action of $G$ on $\bX$. The action on $G\times \bX$ considered below is given by $h(g,x) = (hg,x)$.

The following conditions $(\ref{cond:eqasdim}$--\,$\ref{cond:amenableAction})$ are equivalent and they imply condition $(\ref{cond:almostEqasdim})$. They are all equivalent if $\cF$ is virtually closed (e.g., $\cF=\vcyc$):

\begin{enumerate}
\setcounter{enumi}{-1}
    \item\label{cond:almostEqasdim} for every $m\in\N$ there is $(\bX,\varphi,H)$ and an $m$-long almost $\cF$-cover of $G\times \bX$ of dimension at most $n$;
    \item\label{cond:eqasdim} for every $m\in\N$ there is $(\bX,\varphi,H)$ and an $m$-long $\cF$-cover of $G\times \bX$ of dimension at most $n$;
    \item\label{cond:disjointEqasdim} for every $r\in\N$ there is $(\bX,\varphi,H)$ and disjoint $\cF$-families $(\U^i)$ for $0\leq i \leq n$ of open subsets of $G\times \bX$ such that $\bigcup_i\U^i$ is an $r$-long covering of $G\times \bX$;
    \item\label{cond:epsGMap} for every $\eps>0$ there is $(\bX,\varphi,H)$, an $\cF$-simplicial complex $K$ of dimension~$n$, and a $G$-equivariant continuous map $\phi\colon G\times \bX \to K$, which is ``diagonally $(G,\eps)$-Lipschitz'':
    $$\|\phi(g,x) - \phi(gs^{-1}, f_s(x)) \| \leq \eps\ \ \forall_{s\in S}\ \forall_{f_s \in F_s}\ \forall_{(g,x)\in G\times \bX};$$
    \item\label{cond:amenableAction} for every $\eps>0$ there is $(\bX,\varphi,H)$, an $\cF$-simplicial complex $K$ of dimension~$n$, and a continuous map $\psi\colon \bX \to K$, which is $\eps$-equivariant:
    $$\|\psi(f_s(x)) - s\psi(x) \|\leq \eps\ \ \forall_{s\in S}\ \forall_{f_s \in F_s}\ \forall_{x\in \bX}.$$
\end{enumerate}
\end{thm}

\begin{longRem}\label{def:transferReducible} If a group $G$ satisfies the equivalent conditions ($1$--\,$4$) from the theorem for all finite symmetric subsets $S \subseteq G$ with some additional technical requirements on $\bX$, then $G$ is said to be \emph{transfer reducible} over $\cF$, \cite{flowForCAT0}.
\end{longRem}

The above theorem is stated in terms of \emph{existence} of homotopy actions, however we do not construct spaces $\bX$ and homotopy actions in the proof. Hence, all the equivalences stay true for a fixed $G$-action on a fixed $\bX$ as in the definition of equivariant asymptotic dimension for $G\times \bX$ (Definition \ref{def:eqasdim}).

It was pointed out by M. Bridson that condition (\ref{cond:amenableAction}) of Theorem \ref{thm:characterisations} is very similar to the concept of \emph{amenable action} \cite{onProofs}*{Remark 3.6}. As explained in \cite{coarseFlowSpaces}, this is -- on one hand -- more than an amenable action, where the target space is the whole unit sphere of $\ell_1(G)$, not just an $n$-dimensional complex in it. On the other hand, for $\eqasdim$, space $\ell_1(Y)$ can be build on any $G$-set $Y$ as long as its isotropy groups belong to $\cF$ and for amenable actions we have $Y=G$. In \cite{coarseFlowSpaces}, an action of $G$ on $\bX$ such that $\cF\heqasdim G\times \bX\leq N$ is called \emph{$N$-$\cF$-amenable}.

\begin{proof}[Proof of Theorem \ref{thm:characterisations}]
Implications $(\ref{cond:eqasdim})\implies(\ref{cond:almostEqasdim})$ and $(\ref{cond:disjointEqasdim})\implies(\ref{cond:eqasdim})$ are immediate.

$(\ref{cond:epsGMap})\iff(\ref{cond:amenableAction})$ was suggested in \cite{onProofs}*{Remark 3.7} and holds even for a fixed $\eps$. For the ``if'' part we put $\phi(g,x) = g\psi(x)$. Map $\phi$ is clearly $G$-equivariant and also satisfies the required condition:
$$
\|\phi(g,x) - \phi(gs^{-1},f_s(x))\| = \|g\psi(x) - gs^{-1}\psi(f_s(x))\| = \|s\psi(x) - \psi(f_s(x))\| \leq \eps.
$$
For the ``only if'' part we put $\psi(x)= \phi(1,x)$ and check:
$$\|s\psi(x)-\psi(f_s(x))\| = \|s\phi(1,x) - \phi(1,f_s(x))\| = \|\phi(s,x) - \phi(s s^{-1},f_s(x))\| \leq \eps.$$

$(\ref{cond:epsGMap}) \implies (\ref{cond:disjointEqasdim})$. We will replace $K$ from \eqref{cond:epsGMap} by its 
barycentric subdivision $SK$. The identity map $K\to SK$ is Lipschitz with the constant depending only on $n$. Each vertex of $SK$ corresponds to a subset (simplex) of vertices of $K$, vertices of the same cardinality are not adjacent, and the cardinality of a vertex is clearly preserved under the group action. Moreover, the stabiliser of a vertex in $SK$ is the stabiliser of a simplex in $K$, so it belongs to $\cF$ (in fact also simplex stabilisers belong to $\cF$).

To obtain $(\ref{cond:disjointEqasdim})$, we put $\eps = \frac{1}{(n+1)(r+1)}$ and let $\phi\colon G\times \bX\to SK$ be diagonally $(G,\eps)$-Lipschitz. We define $\U^i = \{\phi^{-1}(S_y) \st y\in V(SK),\ |y|=i+1\}$, where {$i\in\{0,\ldots,n\}$} and $S_y$ is the open star about $y$; that is, $S_y=\{p\in SK \st p(y)>0\}$ (recall that we view $SK$ as a subset of $\ell_1(V(SK))$). The fact that two vertices are non-adjacent is equivalent to disjointness of the respective stars; hence, different elements of $\U^i$ are disjoint. By $G$-equivariance of $\phi$ we get $g\phi^{-1}(S_y) = \phi^{-1}(S_{gy})$, so $\U^i$ is $G$-invariant and the stabiliser of $\phi^{-1}(S_y)$ is the stabiliser of $y$ and thus belongs to $\cF$.

Let now $(g,x)\in G\times \bX$ and $v_0$ be an element $v\in V(SK)$ maximising $\phi(g,x)(v)$. We have $\phi(g,x)(v_0)\geq \frac{1}{n+1}$. Thus, since $\phi$ is diagonally $(G,\eps)$-Lipschitz and $\eps = \frac{1}{(n+1)(r+1)}$, for any $(g',x')\in DB^r_{\varphi,H}(g,x)$ we have $\phi(g',x')(v_0)\geq \frac{1}{n+1} -  \frac{r}{(n+1)(r+1)}>0$. Therefore, $DB^r_{\varphi,H}(g,x)\subseteq \phi^{-1}(S_{v_0})\in \U^i$ for $i = |v_0| - 1$.

$(\ref{cond:eqasdim})\implies (\ref{cond:epsGMap})$. This proof is based on techniques from \cite{Borel}*{Section 3}.

Let $m$ be an integer greater than $\frac{3(n+1)}{\eps}$ and $\U$ be an $m$-long $\cF$-cover of $G\times \bX$. From Lemma \ref{lem:controllingDB} it follows that $I(U)=\{(g,x) \st DB^m_{\varphi, H}(g,x) \subseteq U\}$ is open provided that $U$ is. The family $\mathcal{I}=\{I(U) \st U\in \U\}$ is $G$-invariant (as $hDB^m_{\varphi, H}(g,x)=DB^m_{\varphi, H}(hg,x)$), which means that it looks the same when restricted to $\{g\}\times \bX$ for any $g$. There is a finite cover of $\{1\}\times \bX$ by compact subsets $(W_i)_{i=1}^k$ such that each $W_i$ is contained in some $I(U_i)$ with $U_i\in \U$. Thus, $DB^m_{\varphi,H}(W_i)$ is contained in $U_i$. By compactness (see Lemma \ref{lem:controllingDB}) there exists $\eps_i$ such that $B\left(DB^m_{\varphi,H}(W_i),\ \eps_i\right)\subseteq U_i$.

For $\delta>0$ let us define $DB^{0,\delta}_{\varphi,H}(A) = B(A,\delta)$
 and inductively:
$$DB^{k+1,\delta}_{\varphi,H}(A) = B\left(DB^1_{\varphi,H}\left({DB^{k,\delta}_{\varphi,H}\left(A\right)}\right),\ \delta\right).$$
From Lemma \ref{lem:controllingDB} and induction, it follows that for each $i$ there is $\delta_i$ such that $DB^{m,\delta_i}_{\varphi,H}\left(W_i\right)\subseteq B\left(DB^m_{\varphi,H}(W_i),\ \eps_i\right)\subseteq U_i$. Let $\delta=\min_i \delta_i$ and $\Lambda = \frac{m}{\delta}$.

We will define a $G$-invariant metric $d$ on $G\times \bX$ such that $m$ will be a Lebesgue number of $\U$. Let 
\begin{equation*}
d_0((g,x),(g',x')) =
  \begin{cases}
  \Lambda\cdot d_\bX(x,x'), &\text{if}\ g=g'\text{,}\\
  \max(m,\  \Lambda\cdot \diam\bX), &\text{otherwise.}
  \end{cases}
\end{equation*}
For $y,z\in G\times \bX$ let $d(y,z)$ be equal to the infimum of finite sums $\sum_{i=1}^k\Delta_i((g_j,x_j)_j)$ over finite sequences $(g_j,x_j)_{j=0}^k$ such that $(g_0,x_0)=y$, $(g_k,x_k)=z$, where
\begin{equation*}
\Delta_i((g_j,x_j)_j) =
  \begin{cases}
  1, &\text{if }(g_{i},x_{i}) = (g_{i-1}s^{-1},f_s (x_{i-1}))\ \\
  1, &\text{if } (g_{i-1},x_{i-1}) = (g_{i}s^{-1},f_s (x_{i}))\\
  d_0((g_{i-1},x_{i-1}),(g_{i},x_{i})),
  \end{cases}
\end{equation*}
where $s\in S$ and $f_s\in F_s$
(note that the three cases are not mutually exclusive so we always choose the smallest value).
It is easy to notice that $d$ is symmetric and satisfies the triangle inequality. Moreover, $d_0(y,z)< 1 \iff d(y,z)< 1$ and then they are equal, hence they induce the same topology (the product topology).

If $d(y,z)< m$, there exists a sequence $(g_i,x_i)_{i=0}^k$ joining $y$ and $z$ such that $\sum \Delta_i<m$. In particular, each time $\Delta_i$ is equal to the distance $d_0$ between $(g_{i-1},x_{i-1})$ and $(g_i,x_i)$, this distance is smaller than $m$, meaning that 
\[d_\bX(x_{i-1},x_i) = \Lambda^{-1}\cdot d_0((g_{i-1},x_{i-1}),(g_i,x_i)) < \Lambda^{-1}\cdot m = \delta.\]
The other case happens at most $m$ times. Thus, $z$ belongs to $DB^{m,\delta}_{\varphi,H}(y)$. Since each $y\in G\times \bX$ belongs to some $W_i$ (or a translation of it) and for $W_i$ we have $DB^{m,\delta}_{\varphi,H}(W_i)\subseteq U_i$, we conclude that $m$ is a Lebesgue number of $\U$ with respect to~$d$.

Define $l_U(y) = \min\left(m,\ \sup \{r \st B_d(y,r) \subseteq U\}\right)$. It is clearly $1$-Lipschitz, in particular continuous. Moreover, $d((g,x),(gs^{-1}, f_s(x)))\leq 1$ for $f_s\in F_s$, so we have $|l_U(g,x) - l_U(gs^{-1}, f_s(x))| \leq 1$. Furthermore, by $G$-invariance of $d$, we have $l_U(y)=l_{gU}(gy)$ for any $g\in G$.

We define $\Phi(g,x) = \sum_{U\owns x}l_U(g,x)\cdot\one_U \in \ell_1(\U)$ and $\phi(g,x) = \frac{\Phi(g,x)}{\|\Phi(g,x)\|}$.

From the fact that $l_U^{-1}(0,\infty) \subseteq U$ and that the dimension of $\U$ is at most $n+1$ we conclude that map $\phi$ acquires its values in an $n$-dimensional complex $K\subseteq \ell_1(\U)$. Moreover, since $hU\neq U$ implies $hU\cap U = \emptyset$, we get that $l_U^{-1}(0,\infty) \cap l_{hU}^{-1}(0,\infty) = \emptyset$, so we can assume that $U$ and $hU$ in $K$ are not adjacent. Hence, the stabiliser of a simplex stabilises it pointwise, so it is the intersection of stabilisers of its vertices (corresponding to elements of $\U$) and belongs to $\cF$.

We have to check if $\phi$ is diagonally $(G,\eps)$-Lipschitz. Let $(g,x)\in G \times \bX$, $s\in S$, and $f_s\in F_s$. Without loss of generality:
$$ m \leq\|\Phi(g,x) \| \leq \| \Phi(gs^{-1},f_s(x)) \| \leq \|\Phi(g,x) \| + n+1$$
(in the last inequality we use the fact that $\Phi(gs^{-1},f_s(x))$ has at most $n+1$ points in its support) thus we can write:
\begin{align*}
\big\|\phi(g,x) &- \phi(gs^{-1},f_s(x))\big\| = \left\|\frac{\Phi(g,x)}{\|\Phi(g,x) \|} - \frac{\Phi(gs^{-1},f_s(x))}{\|\Phi(gs^{-1},f_s(x))\|}\right\| 
\\
\leq& \left\|\frac{\Phi(g,x) - \Phi(gs^{-1},f_s(x))}{\|\Phi(g,x) \|}\right\| \\
&+ \left\|\Phi(gs^{-1},f_s(x))\left(\frac{1}{\|\Phi(g,x)\|} - \frac{1}{\|\Phi(gs^{-1},f_s(x))\|}\right)\right\| 
\\
\leq& \frac{2(n+1)}{m} + \left(\frac{\|\Phi(gs^{-1},f_s(x))\|}{\|\Phi(g,x)\|} - 1 \right)
\leq \frac{2(n+1)}{m} + \frac{n+1}{m} < \eps.
\end{align*}

$(\ref{cond:almostEqasdim})\implies (\ref{cond:epsGMap})$ can be proved in the same way as $(\ref{cond:eqasdim})\implies (\ref{cond:epsGMap})$, but we cannot guarantee that simplex stabiliser is a pointwise stabiliser. Simplex stabiliser permutes vertices of the simplex and the kernel of this action is the pointwise stabiliser. This kernel is a finite index subgroup of the stabiliser, hence -- if $\cF$ is virtually closed -- the stabiliser belongs to $\cF$.
\end{proof}

\begin{cor} For a virtually closed $\cF$, the notions \cite{GLnZ}  of groups transfer reducible over $\cF$ and almost transfer reducible over $\cF$ are equivalent.
\end{cor}

The above theorem is formulated for a particular definition of a homotopy action, but should hold for all similar definitions such as \cite{CAT0}*{Definition 2.1}.

\begin{longRem}\label{remark:characterisations} To show $(\ref{cond:epsGMap}) \implies (\ref{cond:eqasdim})$ directly we do not need the continuity of $\phi$. It is enough to assume only that inverse images of stars are open and vertices in the same orbit are not adjacent, and put $\U = \{\phi^{-1}(S_y) \st y\in V(K)\}$.

Conversely, to obtain such a version of (\ref{cond:epsGMap}) from (\ref{cond:eqasdim}) it suffices to define $l_U(g,x)$ as $\max \left\{r\in \{1,\ldots,m\} \st DB^r_{\varphi,H}(g,x) \subseteq U\right\}$.

Note that the implication $(\ref{cond:eqasdim})\implies (\ref{cond:epsGMap})$ (and similar $(\ref{cond:almostEqasdim})\implies (\ref{cond:epsGMap})$) was the only step where we used the metrisability of $\bX$. In Subsection \ref{NoMetric}, we show how to avoid this requirement if we deal with genuine group actions. The analogue of condition (\ref{cond:multiplicity}) from Theorem \ref{thm:asdimCharacterisations} is also provided.
\end{longRem}

\subsection{Theorem \ref{thm:characterisations} without metrisability}\label{NoMetric}

In this subsection we restrict our attention to non-homotopy actions but allow non-metrisable spaces $\bX$.

We already noticed (Remark \ref{remark:characterisations}) that the only part of the proof of Theorem~\ref{thm:characterisations} utilising metrisability was the implication $(\ref{cond:eqasdim})\implies (\ref{cond:epsGMap})$, more precisely, the definition of $l_U$. In Subsection \ref{S:l_UWithoutMetric}, we propose a definition of function $l_U$ in the compact Hausdorff case, which enables proving $(\ref{cond:eqasdim})\implies (\ref{cond:epsGMap})$ in the non-metric setting. Subsection \ref{S:d-disjointness} provides more conditions characterising equivariant asymptotic dimension.
Again, metrisability is not needed for the equivalences to hold.

For the sake of the subsequent reasoning, we will overload our
notation: for $U\subseteq G\times\bX$ and $\alpha>0$ by $\oB(U,\alpha)$ we will denote the set 
$$\oB(U,\alpha) = \bigcup_{(g,x)\in U} \oB(g,\alpha)\times \{x\} = \{(h,x) \st \oB(h,\alpha)\times \{x\} \cap U \neq \emptyset \},$$
and similarly (for $-A$ meaning the complement of $A$):
$$\oB(U,-\alpha) = \{(h,x) \st \oB(h,\alpha)\times \{x\} \subseteq U \} = -\oB(-U,\alpha).$$

Observe that not only $U\mapsto \oB(U,\alpha)$ but also $U\mapsto \oB(U,-\alpha)$ preserves open sets. Indeed, $(h,x)\in \oB(U,-\alpha)$ implies $\oB(h,\alpha)\times \{x\}\subseteq U$. As $U$ is open and $\oB(h,\alpha)$ is finite, $\oB(h,\alpha)\times W \subseteq U$ for some neighbourhood $W$ of $x$ -- and thus $\{h\} \times W \subseteq \oB(U,-\alpha)$.

\subsubsection{\texorpdfstring{Function ${\boldsymbol{l_U}}$ for a compact Hausdorff $\boldsymbol G$-space $\boldsymbol\bX$}{Function l\_U for a compact Hausdorff G-space cl(X)}}\label{S:l_UWithoutMetric}

Let $k > \frac{3(n+1)^2}{\eps}$ and $\U$ be a $k$-$\eqasdim$-covering of $G\times \bX$. Let $\varphi^0_U\colon \bX \to [0,1]$ be a partition of unity subordinate to the restriction of the family $\mathcal I=\{I(U) \st U\in \U \}$ to $\{1\}\times \bX$, where $I(U) = \oB(U,-k)$. We put $\varphi_U = k \cdot \varphi^0_U$.

Let now:
$$l_U^0(g,x) = \max_{h\in G} \left( \varphi_{h^{-1}U}(h^{-1}x) - d_G(g,h) \right).$$
(Note that since $\varphi_U(\cdot)\in[0,k]$, the maximum is in fact taken over a finite set $d_G(g,h)\leq k$.) The positivity of $l_U^0$ for some $(g,x)$ means the existence of $h$ such that $d_G(g,h)<\varphi_{h^{-1}U}(h^{-1}x)$, in particular $\varphi_{h^{-1}U}(h^{-1}x) > 0$ and $d_G(g,h)\leq k$. Thus, point $(1,h^{-1}x)$ belongs to set $I(h^{-1}U)$, which is equivalent to $(h,x)\in I(U)$, which, in turn, implies $(g,x)\in U$ by the definition of $I(U)$. We conclude that $(l_U^0)^{-1}(0,\infty) \subseteq U$.

Clearly, $l_U^0$ also takes values in $[0,k]$ and for each $(g,x)$ there is $U$ such that $l_U^0(g,x) \geq \frac{k}{n+1}$ (it suffices to put $h=g$ and then select $U$). Additionally, it is $G$-invariant:
\begin{multline*}
l_U^0(g,x) = \max_{h\in G}\left( \varphi_{h^{-1}U}(h^{-1}x) - d_G(g,h)\right) =
\\
\max_{h\in G}\left( \varphi_{(jh)^{-1}jU}((jh)^{-1}jx) - d_G(jg,jh)\right) = l_{jU}(jg,jx),
\end{multline*}
where $j\in G$.
Furthermore, it is $1$-Lipschitz with respect to the $G$-coordinate. Indeed, let $s$ be a generator of $G$:
\begin{multline*}
l_U^0(g,x) = 
\max_{h\in G} \left( \varphi_{h^{-1}U}(h^{-1}x) - d_G(g,h)\right) \leq
\\
\max_{h\in G} \left( \varphi_{h^{-1}U}(h^{-1}x) - d_G(gs,h) + 1 \right) = l_U^0(gs,x) + 1.
\end{multline*}
Hence, $l_U(g,x) = l_U^0(g,gx)$ is diagonally $(G,1)$-Lipschitz.

This time, the construction of $l_U^0$ is based on a partition of unity (locally finite), so the continuity of $\phi$ (defined as in the proof of Theorem \ref{thm:characterisations}) is automatic, whereas in Theorem \ref{thm:characterisations} it followed from the Lipschitz property of $\Phi$. The proof of diagonal $(G,\eps)$-Lipschitz property of $\phi$ is analogous.

\subsubsection{\texorpdfstring{$\boldsymbol{d}$-disjointness and $\boldsymbol{r}$-multiplicity}{d-disjointness and r-multiplicity}}\label{S:d-disjointness}

In Theorem \ref{thm:asdimCharacterisations}, there is condition (\ref{cond:multiplicity}) {characterising asymptotic dimension by the existence, for arbitrary $d<\infty$, of a uniformly bounded cover with $d$-multiplicity at most $n+1$. This condition} has no counterpart in Theorem \ref{thm:characterisations} {characterising transfer reducibility}. Moreover condition (\ref{cond:disjoint}) in \ref{thm:asdimCharacterisations} is formulated in terms of $r$-disjoint families, while in \ref{thm:characterisations} (\ref{cond:disjointEqasdim}) we have disjoint families forming a covering with a $G$-Lebesgue number equal to $r$. This lack of analogy is due to the fact that it is not clear how to preserve openness while enlarging sets in order to force large $G$-Lebesgue numbers in the case of homotopy actions.

We fix it in the following proposition. Condition (\ref{cond:disjoint}) of \ref{thm:asdimCharacterisations} has its analogue in condition (\ref{equivariantMultiplicityC}) below and condition (\ref{cond:multiplicity}) of \ref{thm:asdimCharacterisations} is reflected by (\ref{equivariantMultiplicityB}).

We say that a covering of $G\times \bX$ has $(G,d)$-multiplicity $n$ if each set of the form $\oB(g,d) \times \{x\}$ intersects at most $n$ elements of the covering. A family of subsets of $G\times \bX$ is $(G,r)$-disjoint if for any two of its elements $U\neq U'$ we have $\oB(U,r)\cap U'=\emptyset$.

Note that in this subsection we return to the conventions of Section \ref{s:first}, where the action is diagonal, but the notions of $G$-multiplicity, $G$-disjointness etc.\ involve only the first coordinate (formally, the map $\rho(g,x)=(g,g^{-1}x)$ intertwines the two conventions).

\begin{prop}\label{prop:equivariantMultiplicity} Let $\bX$ be a compact Hausdorff space. Conditions $(\ref{equivariantMultiplicityA})$ and $(\ref{equivariantMultiplicityC})$ are equivalent and imply condition $(\ref{equivariantMultiplicityB})$. They are all equivalent if $\cF$ is virtually closed.
\begin{enumerate}
\item\label{equivariantMultiplicityA} $\feqasdim G\times\bX \leq n$;
\item\label{equivariantMultiplicityC} for each $r<\infty$ there exist $(r,G)$-disjoint $\cF$-families $(\mathcal I^i)$ for $0\leq i\leq n$ such that $\bigcup_i \mathcal I^i$ is an open cover of $G\times\bX$;
\item\label{equivariantMultiplicityB} for each $d<\infty$ there exists an open $\cF$-cover $\mathcal I$ of $G\times\bX$ with $(G,d)$-multiplicity at most $n+1$.
\end{enumerate}
\end{prop}
\begin{proof} 
For $(\ref{equivariantMultiplicityA})\implies(\ref{equivariantMultiplicityB})$, we put $\mathcal I = \{\oB(U,-d) \st U\in \U\}$ (we remove the empty set if it appears) for a $d\heqasdim$ family $\U$. Elements of $\mathcal I$ are subsets of elements of $\U$, so the stabilisers may only be smaller (since the definition is equivariant, they are equal).

Let us check the $G$-multiplicity. Consider any $G$-$d$-ball: $\oB(g,d)\times\{x\}$. If it intersects $\oB(U,-d)$ at point $(h,x)$, then, by symmetry of the metric, set $\oB(h,d) \times\{x\}$ contains point $(g,x)$. But we have the inclusion $\oB(h,d) \times\{x\} \subseteq U$ following from the definition of $\oB(U,-d)$, meaning that also set $U$ contains point $(g,x)$. Therefore, the number of sets $\oB(U,-d)$ intersecting $\oB(g,d)\times\{x\}$ does not exceed the number of sets $U$ containing point $(g,x)$, which is bounded by $n+1$.

For $(\ref{equivariantMultiplicityB})\implies(\ref{equivariantMultiplicityA})$ we take $\mathcal I$ for $d=\alpha$ and $\U = \{\oB(V,\alpha) \st V\in\mathcal I \}$. It is clearly $G$-invariant, open and have a $G$-Lebesgue number equal to $\alpha$.

Set $\oB(V,\alpha)$ contains point $(h,x)$ if and only if $\oB(h,\alpha)\times\{x\}$ intersects $V$. Thus, the multiplicity of $\U$ is bounded by the $(G,\alpha)$-multiplicity of $\mathcal I$, so we obtain $\dim\U\leq n$.

We have the equality $g\oB(V,\alpha) = \oB(gV,\alpha)$, so for all elements $g\in G$ such that $g\oB(V,\alpha) = \oB(V,\alpha)$ and for all $(h,x)\in \oB(V,\alpha)$ the set $\oB(h,\alpha)\times\{x\}$ intersects all translates $gV$. Hence, the number of such $gV$ is at most $n+1$ and thus the stabiliser of $\oB(V,\alpha)$ maps into the symmetric group $S(n+1)$ and the kernel is the intersection of stabilisers of sets $gV$. Consequently, the stabiliser of $\oB(V,\alpha)$ has a finite index subgroup from $\cF$.

Hence, if $\cF$ is virtually closed, covering $\U$ satisfies all the conditions for an $\alpha$-$\cF\heqasdim$ covering apart from the fact we do not know whether distinct sets from one orbit are disjoint; it is an almost $\cF$-covering. Theorem \ref{thm:characterisations} (see condition (\ref{cond:almostEqasdim})) in its non-metrisable version shows that for virtually closed $\cF$ it suffices.

For $(\ref{equivariantMultiplicityA})\implies(\ref{equivariantMultiplicityC})$ it is enough to take families $\U^i$ from condition (\ref{cond:disjointEqasdim}) of Theorem \ref{thm:characterisations} and set $\mathcal I^i = \{\oB(U,-d) \st U\in \U^i \}$ as the desired family (we remove empty sets if they appear). Conversely, if families $(\mathcal I^i)$ are $(G,2r+1)$-disjoint, then families $\U^i$ defined by $\U^i=\{\oB(I,r) \st I\in\mathcal I^i\}$ satisfy condition (\ref{cond:disjointEqasdim}) of \ref{thm:characterisations}.
\end{proof}

%
%

\begin{appendices}

\section{Equivariant topological dimension}\label{equivariantTopological}

Recall that \emph{the Lebesgue covering dimension} of a topological space $X$ is the smallest integer~$n$ such that any open covering has a refinement of dimension at most $n$. The number $n$ is sometimes called \emph{the} topological dimension and is denoted by $\dim X$.

If $X$ is an $F$-space for some group $F$, a natural question to ask is whether any $F$-covering has an $F$-refinement of dimension $n$. By an $F$-covering we mean an $\cF$-cover, where $\cF$ is the family of all subgroups of $F$. In other words: the covering is $F$-invariant and two distinct elements of an orbit are disjoint.

The question was 
 answered in positive in \cite{refinements} for a finite group $F$ acting on a metric space by isometries. This made the bound in Propositions 3.2, 3.3 of \cite{coversForHyperbolic} independent of the order of the group $F$. In \cites{coversForHyperbolic, GLnZ} a bound on the orders of finite subgroups $F$ of a group was needed. Due to the above improvement, a proof of the Farrell--Jones conjecture became possible in a situation where no such bound exists \cite{GLn}.

We will prove that the assumption that the group $F$ acting on the space is finite, is superfluous. It is enough to assume properness of the action. Moreover, our argument remains true for $\cF$-covers with arbitrary $\cF$.

%
%

\subsection[Dimension theory – auxiliaries]{Dimension theory -- auxiliaries}

Recall some definitions and facts from dimension theory after \cite{dimension}.

\begin{thm}[{\cite{dimension}*{9.2.16}}]\label{thm:openSurjection} Let $f\colon X\to Y$ be a continuous open surjection of metrisable spaces. If every fibre $f^{-1}(y)$ is finite, then $\dim X = \dim Y$.
\end{thm}

\begin{defi}[{\cite{dimension}*{5.1.1}}] \emph{The local dimension}, $\locdim X$, of a topological space $X$ is defined as follows. If $X$ is empty, then $\locdim X = -1$. Otherwise, $\locdim X$ is the smallest integer $n$ such that for every point $x\in X$ there is an open set $U\owns x$ such that $\dim \myol U\leq n$. If there is no such $n$, then $\locdim X = \infty$.
\end{defi}

\begin{thm}[{\cite{dimension}*{5.3.4}}] If $X$ is a metric space, then $\locdim X = \dim X$.
\end{thm}

\begin{cor}\label{cor:dimOpenMonotonicity} If $V$ is an open subset of a metric space $X$, then $\dim V\leq \dim X$.
\end{cor}
\begin{proof} 
By the above theorem, it is enough to prove $\locdim V\leq \dim X$. Consider $x\in V$. There is an open neighbourhood $V_x\owns x$ such that $\myol V_x\subseteq V$. Hence, since the dimension of a closed subset never exceeds the dimension of the space, we obtain:
$$\locdim V = \sup_{x\in V} \inf_{V \supseteq U\owns x} \dim \myol U \leq \sup_{x\in V} \dim \myol V_x \leq \dim X$$
(where $U$ is open and the closure of $U$ is taken in $V$).
\end{proof}

\begin{cor}\label{cor:locdimSimplification} In the case of metric spaces, there is no need for taking closures of neighbourhoods in the definition of local dimension (it is enough to consider open neighbourhoods and calculate their dimension).
\end{cor}
\begin{proof} Fix $x\in X$. It suffices to check the equality $\inf_{U\owns x}\dim \myol U = \inf_{U\owns x}\dim U$, where $U$ are open neighbourhoods of $x$. Let $U_x$ be an open neighbourhood of $x$ such that $\myol U_x$ has the smallest possible dimension. Then dimension of $U_x$ -- by Corollary \ref{cor:dimOpenMonotonicity} -- is no larger. On the other hand, if there is an open neighbourhood $V$ of $x$ such that $\dim V < \dim \myol U_x$, then there would be an open neighbourhood $W$ such that $\myol W \subseteq V$ and thus $\dim \myol W \leq \dim V < \dim \myol U_x$, contradicting the minimality of $\dim \myol U_x$.
\end{proof}

\begin{prop}\label{prop:dimAsSup} The dimension of a metric space $X$ is equal to the supremum of dimensions of its open subsets. It is enough to consider the supremum over any open cover of $X$.
\end{prop}
\begin{proof} Let $\U$ by any open covering of $X$. By Corollary \ref{cor:dimOpenMonotonicity}, the dimension of $X$ is no smaller than dimensions of its open subsets, thus $\dim X \geq \sup_{U\in \U} \dim U$. On the other hand, it equals the local dimension, which is equal -- by Corollary \ref{cor:locdimSimplification} -- to the supremum over points of infima over open neighbourhoods of their dimensions. But clearly, we have the inequalities:
\[\dim X =
\sup_x \inf_{U\owns x} \dim U \leq
\sup_x \inf_{\U\owns U\owns x} \dim U \leq
\sup_{x,\ \U\owns U\owns x} \dim U =
\sup_{U\in \U} \dim U.\qedhere\]
\end{proof}

%
%

\subsection{Equivariant refinements}

The following proposition strengthens \cite{refinements}*{Corollary 2.5}.

\begin{prop}\label{prop:quotientDimension} Let $(X,d)$ be a metric space with an isometric proper action of a group~$G$. Then $\dim X / G = \dim X$.
\end{prop}
\begin{proof} We can fix a pseudometric on the quotient space: $$d'([x], [x']) = \inf_{g,g'\in G} d(gx, g'x').$$ The action is isometric, so it is equal to $ \inf_{g \in G} d(gx, x')$. If $[x]\neq [x']$, then -- by properness of the action -- there is no infinite sequence $g_nx$ {converging} to $x'$ and thus $d'([x], [x']) > 0$. Therefore, $X/G$ is a metric space (it is easy to check that the quotient topology and the metric topology agree).

Let $x \in X$. Similarly as above, there is $\eps=\eps(x)>0$ such that $B(x,2\eps)$ is disjoint with all the other elements of the orbit $Gx$. Consequently, $B(x,\eps)$ is disjoint with its translates and has a finite stabiliser $S$ (the one of $x$).

Denote by $f$ the restriction of the quotient map $q\colon X\to X/G$ to $B(x,\eps)$. For $x'\in B(x,\eps)$ and $y'=f(x')$, the fibre $f^{-1}(y')$ is equal to $S x'$ and thus finite. Clearly $f$ is an open surjection onto its (open) image, so Theorem \ref{thm:openSurjection} applies: $\dim f\left(B(x,\eps)\right) = \dim B(x,\eps)$.

Using the openness and the surjectivity again, we notice that the family
$$\left\{q\left(B\big(x,\eps(x)\big)\right) \st x\in X\right\}$$
is an open covering of $X / G$. With Proposition \ref{prop:dimAsSup} we conclude:
\[\dim X / G = \sup \dim q\left(B\left(x,\eps(x)\right)\right) = \sup \dim B(x,\eps(x))  = \dim X.\qedhere\]
\end{proof}

Finally, we can prove a version of \cite{refinements}*{Proposition 2.6}.

\begin{prop}\label{prop:refinements} Let $X$ be a metric space with an isometric proper action of a group~$G$ and $\dim X=n$. Any open $\cF$-cover $\U$ of $X$ has an open $\cF$-refinement $\W$ with dimension at most~$n$.
\end{prop}

\begin{proof} Denote the quotient map by $q$. By Proposition \ref{prop:quotientDimension}, we know that the open covering $\{q(U) \st U\in \U \}$ of $X / G$ has a refinement $\V$ of dimension at most $n$.

Clearly $q^{-1}(V)$ for $V\in \V$ is $G$-invariant, in particular it is a $G$-subset. The covering $\{q^{-1}(V) \st V\in \V \}$ has the same dimension as $\V$.

In order to obtain the required refinement of $\U$, it is enough to divide each $q^{-1}(V)$ into appropriate disjoint parts. Note that division into disjoint parts does not increase the dimension of a covering. Let $U_V$ be such an element of $\U$ that $V\subseteq q(U_V)$. Then clearly:
$$q^{-1}(V)\subseteq q^{-1}(q(U_V)) = \bigsqcup_{[g]\in G/S}gU_V,$$
where $S$ is the stabiliser of $U_V$. The required division is $\bigsqcup_{[g]} gU_V\cap q^{-1}(V).$ The covering $\W = \{g U_V \cap q^{-1}(V) \st V\in \V,\ g\in G\}$ is clearly a $G$-covering and refines $\U$. Moreover, if $\U$ is an $\cF$-cover, then $\W$ also is, as the stabiliser of $U_V\cap q^{-1}(V)$ is the same as the stabiliser of $U_V$.
\end{proof}

\end{appendices}

%
%

\section*{Acknowledgements}

This work is based on the author's Master's thesis defended in September 2014 at the Faculty of Mathematics, Informatics and Mechanics of the University of Warsaw. The author is grateful to the supervisor of the thesis, Piotr Nowak, for his guidance and help. The author wishes to thank Rufus Willett for explaining the quantitative relations between equivariant and classical asymptotic dimension. The author especially thanks the referee and Łukasz Garncarek for very helpful remarks that greatly improved the manuscript.

%
%

\end{document}